\documentclass[12pt]{article}
\usepackage{fullpage,amsthm,amsmath,amssymb,xspace,tikz,url,caption,subcaption}

\newcounter{thmctr}
\newtheorem{thm}[thmctr]{Theorem}
\newtheorem{lemma}[thmctr]{Lemma}
\newtheorem{prop}[thmctr]{Proposition}
\newtheorem{cor}[thmctr]{Corollary}
\newtheorem{problem}[thmctr]{Problem}
\newtheorem*{definition}{Definition}
\theoremstyle{definition}
\newtheorem*{constr}{Construction}
\theoremstyle{plain}

\newcommand{\npmu}{$(n,p,\mu)$\xspace}
\newcommand{\dense}{$(p,\mu)$-dense\xspace}
\newcommand{\npmua}{$(n,p,\mu,\alpha)$\xspace}

\DeclareMathOperator{\inj}{\text{inj}}

\tikzstyle{vertex}=[circle,fill=black,inner sep=2pt]

\newcommand{\dhruvuni}{University of Illinois at Chicago \\ mubayi@math.uic.edu}
\newcommand{\johnuni}{University of Illinois at Chicago \\ lenz@math.uic.edu}
\newcommand{\dhruvfoot}{\footnote{Research supported in part by  NSF Grants 0969092 and 1300138.}}
\newcommand{\johnfoot}{\footnote{Research partly supported by NSA Grant H98230-13-1-0224.}}

\title{Perfect Packings in Quasirandom Hypergraphs I}
\author{John Lenz \johnfoot \\ \johnuni \and Dhruv Mubayi \dhruvfoot \\ \dhruvuni}

\begin{document}

\maketitle

\begin{abstract}
  Let $k \geq 2$ and $F$ be a linear $k$-uniform hypergraph with $v$ vertices.  We prove that if $n$
  is sufficiently large and $v|n$, then every quasirandom $k$-uniform hypergraph on $n$ vertices
  with constant edge density and minimum degree $\Omega(n^{k-1})$ admits a perfect $F$-packing.  The
  case $k = 2$ follows immediately from the blowup lemma of Koml\'os, S\'ark\"ozy, and Szemer\'edi.
  We also prove positive results for some nonlinear $F$ but at the same time give counterexamples
  for rather simple $F$ that are close to being linear.  Finally, we address the case when the
  density tends to zero, and prove (in analogy with the graph case) that sparse quasirandom
  3-uniform hypergraphs admit a perfect matching as long as their second largest eigenvalue is
  sufficiently smaller than the largest eigenvalue.
\end{abstract}

\section{Introduction} 
\label{sec:intro}

A $k$-uniform hypergraph $H$ ($k$-graph for short) is a collection of $k$-element subsets (edges) of
a vertex set $V(H)$. For a $k$-graph $H$ and a subset $S$ of vertices of size at most $k-1$, let
$d(S)=d_H(S)$ be the number of subsets of size $k - |S|$ that when added to $S$ form a edge of
$H$.  The \emph{minimum degree} of $H$, written $\delta(H)$, is the minimum of $d(\{s\})$ over all
vertices $s$.  The \emph{minimum $\ell$-degree} of $H$, written $\delta_\ell(H)$, is the minimum of
$d(S)$ taken over all $\ell$-sets of vertices.  The \emph{minimum codegree} of $H$ is the minimum
$(k-1)$-degree.  Let $K_t^k$ be the complete $k$-graph on $t$ vertices.

Let $G$ and $F$ be $k$-graphs.  We say that $G$ has a \emph{perfect $F$-packing} if the vertex set
of $G$ can be partitioned into copies of $F$.  An important result of Hajnal and
Szemer\'edi~\cite{pp-hajnal70} states that if $r$ divides $n$ and the minimum degree of an
$n$-vertex graph $G$ is at least $(1-1/r)n$, then $G$ has a perfect $K_r$-packing.  Later Alon and
Yuster~\cite{pp-alon96} conjectured that a similar result holds for any graph $F$ instead of just
cliques, with the minimum degree of $G$ depending on the chromatic number of $F$.  This was proved
by Koml\'os-S\'ark\"ozy-Szemer\'edi~\cite{pp-komlos01} by using the Regularity Lemma and Blow-up
Lemma.  Later, K\"uhn and Osthus~\cite{pp-kuhn09} found the minimum degree threshold for perfect
$F$-packings up to a constant; the threshold either comes from the chromatic number of $F$ or the
so-called critical chromatic number of $F$.

In the past decade there has been substantial interest in extending this result to $k$-graphs.
Nevertheless, the simplest case of determining the minimum codegree threshold that guarantees a
perfect matching was settled only recently by R\"odl-Ruci\'nski-Szemer\'edi~\cite{pp-rodl09}.  Since
then, there are a few results for codegree thresholds for packing other small
3-graphs~\cite{pp-czgrinow13, pp-keevash13, pp-kuhn06-cherry, pp-lo-kfour, pp-lo13, pp-pikhurko08,
pp-treglown09, pp-treglown13}.  For $\ell$-degrees with $\ell < k/2$ (in particular the minimum
degree), much less is known.  After work by many researchers~\cite{pp-han09, pp-kahn11, pp-kahn13,
pp-kuhn13, pp-kuhn13-fractional, pp-markstrom11}, still only the degree threshold for
$K_3^3$-packings and $K_4^3$-packings are known ($\frac{5}{9}$ and $\frac{37}{64}$ respectively).
For $m \geq 5$ and $k \geq 4$ the packing degree threshold for $K_m^k$ is open
(\cite{pp-kuhn13-fractional} contains the current best bounds).

A key ingredient in the proofs of most of the above results are specially designed random-like
properties of $k$-graphs that imply the existence of perfect $F$-packings.  There is a rather
well-defined notion of quasirandomness for graphs that originated in early work of
Thomason~\cite{qsi-thomason87,qsi-thomason87-2} and Chung-Graham-Wilson~\cite{qsi-chung89} which
naturally generalizes to $k$-graphs. Our main focus in this paper is on understanding when perfect
$F$-packings exist in quasirandom hypergraphs.  The basic property that defines quasirandomness is
uniform edge-distribution, and this extends naturally to hypergraphs.  Let $v(H) = |V(H)|$.

\begin{definition}
  Let $k \geq 2$, let $0<\mu,p<1$, and let $H$ be a $k$-graph.  We say that $H$ is $(p,\mu)$-dense
  if for all $X_1, \dots, X_k \subseteq H$,
  \begin{align*}
    e(X_1,\dots,X_k) \geq p |X_1| \cdots |X_k| - \mu n^k,
  \end{align*}
  where $e(X_1,\dots,X_k)$ is the number of $(x_1,\dots,x_k) \in X_1 \times \cdots \times X_k$ such
  that $\{x_1,\dots,x_k\} \in H$ (note that if the $X_i$s overlap an edge might be counted more than
  once).  Say that $H$ is an $(n,p,\mu)$ $k$-graph if $H$ has $n$ vertices and is $(p,\mu)$-dense.
  Finally, if $0 < \alpha < 1$, then an $(n,p,\mu)$ $k$-graph is an $(n,p,\mu,\alpha)$ $k$-graph if
  its minimum degree is at least $\alpha \binom{n}{k-1}$.
\end{definition}

The $F$-packing problem for quasirandom graphs with constant density has been solved implicitly by
Koml\'os-S\'ark\"ozy-Szemer\'edi~\cite{pp-komlos97} in the course of developing the Blow-up Lemma.

\begin{thm} \label{kss} {\bf (Koml\'os-S\'ark\"ozy-Szemer\'edi~\cite{pp-komlos97})}
  Let $0 < \alpha, p < 1$ be fixed and let $F$ be any graph.  There exists an $n_0$ and $\mu > 0$
  such that if $H$ is any $(n,p,\mu,\alpha)$ $2$-graph where $n \geq n_0$, $v(F) | n$ then $H$ has a
  perfect $F$-packing.
\end{thm}

Note that the condition on minimum degree is required, since if the condition ``$\delta(H) \geq
\alpha n$'' in Theorem~\ref{kss} is replaced by ``$\delta(H) \geq f(n)$'' for any choice of $f(n)$
with $f(n) = o(n)$, then there exists the following counterexample. Take the disjoint union of the
random graph $G(n,p)$ and a clique of size either $\left\lceil f(n) \right\rceil + 1$ or
$\left\lceil f(n)\right\rceil+2$ depending on which is odd.  The minimum degree is at least $f(n)$,
there is no perfect matching, and the graph is still $(p,\mu)$-dense.  Because of the use of the
regularity lemma, the constant $n_0$ in Theorem~\ref{kss} is an exponential tower in $\mu^{-1}$.  We
extend Theorem~\ref{kss} to a variety of $k$-graphs.  In the process, we also reduce the size of
$n_0$ for all $2$-graphs. A basic problem in this area that naturally emerges is the following.

\begin{problem} \label{prob}
  For which $k$-graphs $F$ does the following hold: for all $0 < p,\alpha < 1$, there is some $n_0$
  and $\mu$ so that if $H$ is an $(n,p,\mu,\alpha)$ $k$-graph with $n \geq n_0$ and $v(F)|n$, then
  $H$ has a perfect $F$-packing.
\end{problem}

Unlike the graph case, most $F$ will not satisfy Problem~\ref{prob}.  Indeed, R\"odl observed that
for all $\mu > 0$ and there is an $n_0$ such that for $n \geq n_0$, an old construction of Erd\H os and
Hajnal~\cite{ram-erdos72} produces an $n$-vertex 3-graph which is $(\frac{1}{4},\mu)$-dense and has
no copy of $K_4^3$.  In a forthcoming paper we will show that a stronger notion of quasirandomness
suffices to perfectly pack all $F$.

A hypergraph is \emph{linear} if every two edges share at most one vertex.  For a $k$-graph $H$,
Kohayakawa-Nagle-R\"odl-Schacht~\cite{hqsi-kohayakawa10} recently proved an equivalence between
$(|H|/\binom{n}{k},\mu)$-dense and the fact that for each linear $k$-graph $F$, the number of
labeled copies of $F$ in $H$ is the same as in the random graph with the same density.  This leads
naturally to the question of whether Problem~\ref{prob} has a positive answer for linear $k$-graphs,
and our first result shows that this is the case.

\begin{thm} \label{thm:linearpacking}
  Let $k \ge 2$, $0 < \alpha, p < 1$, and let $F$ be a linear $k$-graph.  There exists an $n_0$ and
  $\mu > 0$ such that if $H$ is an $(n,p,\mu,\alpha)$ $k$-graph where $n \geq n_0$ and $v(F)|n$, then
  $H$ has a perfect $F$-packing.
\end{thm}

We restrict our attention only to 3-graphs now although the concepts extend naturally to larger $k$.
Define a $3$-graph to be \emph{$(2+1)$-linear} if its edges can be ordered as $e_1, \ldots, e_q$
such that each $e_i$ has a partition $s_i \cup t_i$ with $|s_i|=2, |t_i|=1$ and for every $j<i$ we
have $e_j \cap e_i \subseteq s_i$ or $e_j \cap e_i \subseteq t_i$.  In words, every edge before
$e_i$ intersects $e_i$ in a subset of $s_i$ or of $t_i$.  Clearly every linear 3-graph is
$(2+1)$-linear, but the converse is false.  Keevash's~\cite{design-keevash14} recent proof of the
existence of designs and our recent work on quasirandom properties of
hypergraphs~\cite{hqsi-lenz-quasi12, hqsi-lenz-quasi12-nonregular, hqsi-lenz-poset12} use a
quasirandom property distinct from $(p,\mu)$-dense that Keevash calls \emph{typical} and we call
\emph{$(2+1)$-quasirandom} (although the properties are essentially equivalent).  These properties
imply that the count of all $(2+1)$-linear $3$-graphs in a typical $3$-graph is the same as in the
random $3$-graph (see~\cite{hqsi-lenz-quasi12,hqsi-lenz-quasi12-nonregular}).
 
Thus a natural direction in which to extend Theorem~\ref{thm:linearpacking} is to the family of
$(2+1)$-linear $3$-graphs and we begin this investigation with some of the smallest such $3$-graphs.
A \emph{cherry} is the 3-graph comprising two edges that share precisely two vertices - this is the
``simplest" non-linear hypergraph. A more complicated $(2+1)$-linear 3-graph is $C_4(2+1)$ which has
vertex set $\{1,2,3,4,a,b\}$ and edge set $\{12a, 12b, 34a, 34b\}$.  The importance of $C_4(2+1)$
lies in the fact that $C_4(2+1)$ is forcing for the class of all $(2+1)$-linear $3$-graphs.  This
means that if $F$ is a $(2+1)$-linear $3$-graph and $p,\epsilon > 0$ are fixed, there is $n_0$ and
$\delta > 0$ so that if $n \geq n_0$ and $H$ is an $n$-vertex $3$-graph with $p\binom{n}{3}$ edges
and $(1\pm\delta)p^4n^6$ labeled copies of $C_4(2+1)$, then the number
of labeled copies of $F$ in $H$ is $(1\pm\epsilon)p^{|F|}n^{v(F)}$ (see~\cite{hqsi-lenz-quasi12,
hqsi-lenz-quasi12-nonregular}).
 
\begin{figure}[h] 
  \center
  \subcaptionbox*{Cherry}{%
  \begin{tikzpicture}
    \node (x) at (0,0) [vertex] {};
    \node (y) at (1,0) [vertex] {};
    \node (z1) at (2,0.5) [vertex] {};
    \node (z2) at (2,-0.5) [vertex] {};

    \draw (x) ..controls (y) .. (z1);
    \draw (x) ..controls (y) .. (z2);
  \end{tikzpicture}
  }
  \hspace{1cm}
  \subcaptionbox*{$C_4(2+1)$}{%
  \hspace{0.7cm}
  \begin{tikzpicture}
    \node (x1) at (-0.1,0.1) [vertex] {};
    \node (x2) at (0.1,-0.1) [vertex] {};
    \node (y) at (1, 0) [vertex] {};
    \node (z1) at (0.9,1.1) [vertex] {};
    \node (z2) at (1.1,0.9) [vertex] {};
    \node (w) at (0,1) [vertex] {};

    \draw[rounded corners=5pt] (0.5,-0.25) -- (-0.4,-0.4) -- (-0.4,0.4) -- (1.4,0.14) --
      (1.4,-0.14) -- (0.5,-0.25);

    \begin{scope}[yshift=1cm, rotate around={180:(0.5,0)}]
      \draw[rounded corners=5pt] (0.5,-0.25) -- (-0.4,-0.4) -- (-0.4,0.4) -- (1.4,0.14) --
        (1.4,-0.14) -- (0.5,-0.25);
    \end{scope}

    \begin{scope}[rotate around={90:(0,0)}]
      \draw[rounded corners=5pt] (0.5,-0.25) -- (-0.4,-0.4) -- (-0.4,0.4) -- (1.4,0.14) --
      (1.4,-0.14) -- (0.5,-0.25);
    \end{scope}

    \begin{scope}[rotate around={-90:(1,0)}]
      \draw[rounded corners=5pt] (0.5,-0.25) -- (-0.4,-0.4) -- (-0.4,0.4) -- (1.4,0.14) --
      (1.4,-0.14) -- (0.5,-0.25);
    \end{scope}
  \end{tikzpicture}
  \hspace{0.7cm}
  }
  \caption{Two $3$-graphs}
\end{figure}
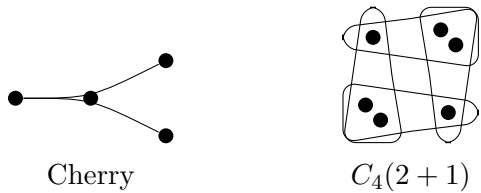 
 	
\begin{thm} \label{thm:three-unif-packings}
  Let $0 < \alpha, p < 1$.  There exists an $n_0$ and $\mu > 0$ such that if $H$ is an
  $(n,p,\mu,\alpha)$ $3$-graph where $n \geq n_0$, then $H$ has a perfect cherry-packing if
  $4|n$ and a perfect $C_4(2+1)$ packing if $6|n$.
\end{thm}

One might speculate that Theorem~\ref{thm:three-unif-packings} can be extended to the collection of
all $(2+1)$-linear $F$ or to the collection of all $3$-partite $F$.  However, our next result shows
that this is not the case and that solving Problem~\ref{prob} will be a difficult project.  If $x$
is a vertex in a $3$-graph $H$, the \emph{link} of $x$ is the graph with vertex set $V(H) \setminus
\{x\}$ and edges those pairs who form an edge with $x$.

\begin{thm} \label{thm:nopacking}
  Let $F$ be any $3$-graph with an even number of vertices such that there exists a
  partition of the vertices of $F$ into pairs such that each pair has a common edge in their links.
  Then for any $\mu > 0$, there exists an $n_0$ such that for all $n \geq n_0$, there exists a
  $3$-graph $H$ such that
  \begin{itemize}
    \item $|H| = \frac{1}{8} \binom{n}{3} \pm \mu n^3$,
    \item $H$ is $(\frac{1}{8},\mu)$-dense,
    \item $\delta(H) \geq (\frac{1}{8} - \mu) \binom{n}{2}$,
    \item $H$ has no perfect $F$-packing.
  \end{itemize}
\end{thm}

Two examples of $3$-graphs $F$ that satisfy the conditions of Theorem~\ref{thm:nopacking} are the
complete 3-partite 3-graph $K_{2,2,2}$ with parts of size two and the following $(2+1)$-linear
hypergraph. A \emph{cherry 4-cycle} is the $(2+1)$-linear 3-graph with edge set $\{123, 124, 345,
\linebreak[1] 346, 567, 568, 781, 782\}$.

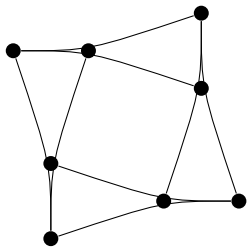
\begin{figure}[ht] 
  \center
  \begin{tikzpicture}
    \node (x1) at (-0.5,2) [vertex] {};
    \node (x2) at (0.5,2) [vertex] {};
    \node (y1) at (2,2.5) [vertex] {};
    \node (y2) at (2,1.5) [vertex] {};
    \node (z1) at (2.5,0) [vertex] {};
    \node (z2) at (1.5,0) [vertex] {};
    \node (w1) at (0,-0.5) [vertex] {};
    \node (w2) at (0,0.5) [vertex] {};

    \draw (x1) ..controls (x2) .. (y1);
    \draw (x1) ..controls (x2) .. (y2);
    \draw (y1) ..controls (y2) .. (z1);
    \draw (y1) ..controls (y2) .. (z2);
    \draw (z1) ..controls (z2) .. (w1);
    \draw (z1) ..controls (z2) .. (w2);
    \draw (w1) ..controls (w2) .. (x1);
    \draw (w1) ..controls (w2) .. (x2);
  \end{tikzpicture}
  \caption{cherry $4$-cycle}
\end{figure} 

It is straightforward to see that Theorem~\ref{thm:nopacking} applies to the cherry $4$-cycle.
Therefore one cannot hope that Theorem~\ref{thm:linearpacking} holds for all $(2+1)$-linear or
3-partite $F$.

Our final result considers the situation when the density is not fixed and goes to zero. Here the
notion of quasirandom is measured by spectral gap.  It is a folklore result that large spectral gap
guarantees a perfect matching in graphs. For hypergraphs, there are several definitions of
eigenvalues. We will use the definitions that originated in the work of Friedman and
Wigderson~\cite{ee-friedman95,ee-friedman95-2} for regular hypergraphs. The definition for all
hypergraphs can be found in \cite[Section 3]{hqsi-lenz-quasi12} where we specialize to $\pi = 1 +
\dots + 1$.  That is, let $\lambda_1(H) = \lambda_{1,1+\dots+1}(H)$ and let $\lambda_2(H) =
\lambda_{2,1+\dots+1}(H)$, where both $\lambda_{1,1+\dots+1}(H)$ and $\lambda_{2,1+\dots+1}(H)$ are
as defined in Section~3 of~\cite{hqsi-lenz-quasi12}.  The only result about eigenvalues that we will
require is Proposition~\ref{prop:expandermixing}, which is usually called the Expander Mixing
Lemma~\cite[Theorem 4]{hqsi-lenz-quasi12} (see also~\cite{ee-friedman95, ee-friedman95-2}).

\begin{thm} \label{thm:sparse}
  For every $\alpha > 0$, there exists $n_0$ and $\gamma > 0$ depending only on $\alpha$
  such that the following holds.  Let $H$ be an $n$-vertex $3$-graph where $3|n$ and
  $n \geq n_0$.  Let $p = 6|H|/n^3$ and assume that $\delta_2(H) \geq \alpha p n$ and
  \begin{align*}
    \lambda_2(H) \leq \gamma p^{16} n^{3/2}.
  \end{align*}
  Then $H$ contains a perfect matching.
\end{thm}

Let $\Delta_2(H)$ be the \emph{maximum codegree of a $3$-graph $H$}, i.e.\ the maximum of $d(S)$
over all $2$-sets $S\subseteq V(H)$.  If $\Delta_2(H) \leq c pn$ then $\lambda_1(H) \leq c'pn^{3/2}$
where $c'$ is a constant depending only on $c$.  This implies the following corollary.

\begin{cor}
  For every $\alpha > 0$, there exists $n_0$ and $\gamma > 0$ depending only on $\alpha$
  such that the following holds.  Let $H$ be an $n$-vertex $3$-graph where $3|n$ and
  $n \geq n_0$.  Let $p = |H|/\binom{n}{3}$ and assume that $\delta_2(H) \geq \alpha pn$,
  $\Delta_2(H) \leq \frac{1}{\alpha} p n$, and
  \begin{align*}
    \lambda_2(H) \leq \gamma p^{15} \lambda_1(H).
  \end{align*}
  Then $H$ contains a perfect matching.
\end{cor}

The third largest eigenvalue of a graph is closely related to its matching number (see
e.g.~\cite{pp-cioaba09}), but currently we do not know the ``correct'' definition of $\lambda_3$ for
hypergraphs.  It would be interesting to discover a definition of $\lambda_3$ for $k$-graphs which
extends the graph definition and for which a bound on
$\lambda_3$ forces a perfect matching.

The remainder of this paper is organized as follows.  In Section~\ref{sec:tools} we will develop the
tools neccisary for our proofs, including extensions of the absorbing technique and various
embedding lemmas.  Then in Section~\ref{sec:existance} we will use these to prove
Theorem~\ref{thm:linearpacking} (Section~\ref{sec:linear}) and Theorem~\ref{thm:three-unif-packings}
(Sections~\ref{sec:Cherry} and~\ref{sec:cycle}).  Section~\ref{sec:constr} contains the
construction proving Theorem~\ref{thm:nopacking} and Section~\ref{sec:Sparse} has the proof of the
sparse case, Theorem~\ref{thm:sparse}.

\section{Tools} 
\label{sec:tools}

In this section, we state and prove several lemmas and propositions that we will need; our main tool is the
absorbing technique of R\"odl-Ruci\'nski-Szemer\'edi~\cite{pp-rodl09}.

\begin{definition}
  Let $F$ and $H$ be $k$-graphs and let $A, B \subseteq V(H)$.  We say that $A$
  \emph{$F$-absorbs} $B$ or that $A$ is an \emph{$F$-absorbing set} for $B$ if both $H[A]$ and $H[A
  \cup B]$ have perfect $F$-packings.  When $F$ is a single edge, we say that $A$
  \emph{edge-absorbs} $B$.
\end{definition}

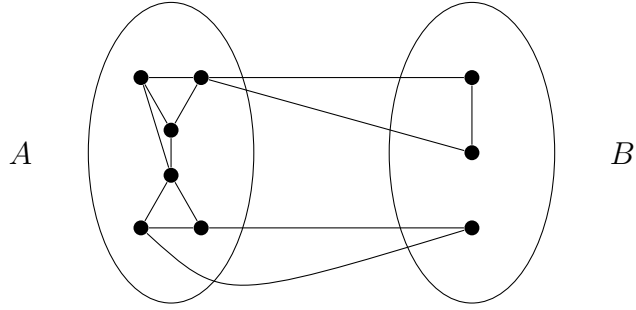
\begin{figure}[ht] 
  \center
  \begin{tikzpicture}
    \draw ellipse (1.1cm and 2cm);
    \draw (4,0) ellipse (1.1cm and 2cm);
    \node at (-2,0) {$A$};
    \node at (6,0) {$B$};

    \node (a1) at (-0.4,1) [vertex] {};
    \node (a2) at (0.4,1) [vertex] {};
    \node (a3) at (0,0.3) [vertex] {};
    \draw (a1) -- (a2) -- (a3) -- (a1);

    \node (a4) at (-0.4,-1) [vertex] {};
    \node (a5) at (0.4,-1) [vertex] {};
    \node (a6) at (0,-0.3) [vertex] {};
    \draw (a4) -- (a5) -- (a6) -- (a4);
    \draw (a3) -- (a6) -- (a1);

    \node (b1) at (4,1) [vertex] {};
    \node (b2) at (4,0) [vertex] {};
    \node (b3) at (4,-1) [vertex] {};
    \draw (a2) -- (b1) -- (b2) -- (a2);
    \draw (a5) -- (b3) .. controls (0.7,-2) .. (a4);
  \end{tikzpicture}
  \caption{$A$ $K_3$-absorbs $B$}
\end{figure} 

\begin{definition}
  Let $F$ and $H$ be $k$-graphs, $\epsilon > 0$, $a$ and $b$ be multiples of $v(F)$,
  $\mathcal{A} \subseteq \binom{V(H)}{a}$, and $\mathcal{B} \subseteq \binom{V(H)}{b}$.  We say
  that $H$ is \emph{$(\mathcal{A},\mathcal{B},\epsilon,F)$-rich} if for all $B \in \mathcal{B}$
  there are at least $\epsilon n^a$ sets in $\mathcal{A}$ which $F$-absorb $B$.  If $\mathcal{A} =
  \binom{V(H)}{a}$, we abbreviate this to $(a,\mathcal{B},\epsilon,F)$-rich and if both $\mathcal{A}
  = \binom{V(H)}{a}$ and $\mathcal{B} = \binom{V(H)}{b}$, we abbreviate this to
  $(a,b,\epsilon,F)$-rich.
\end{definition}

The following proposition is one of the main results of this section; the proof appears in
Section~\ref{sub:proof:richimpliespacking}.

\begin{prop} \label{prop:richimpliespacking}
  Fix $0 < p < 1$, let $F$ be a $k$-graph such that $F$ is either linear or $k$-partite, and
  let $a$ and $b$ be multiples of $v(F)$.  For any $\epsilon > 0$, there exists an $n_0$ and $\mu
  > 0$ such that the following holds. If $H$ is an $(a,b,\epsilon,F)$-rich, \npmu $k$-graph where
  $v(F)|n$, then $H$ has a perfect $F$-packing.
\end{prop}

We will actually need a slight extension of Proposition~\ref{prop:richimpliespacking} for some of
our results and this requires an additional definition.

\begin{definition}
  Let $\zeta > 0$, $t$ be any integer, $H$ be a $3$-graph, and $B \subseteq V(H)$ with $|B| = 2t$.
  We say that $B$ is \emph{$\zeta$-separable} if there exists a partition of $B$ into $B_1, \dots,
  B_t$ such that for all $i$ $|B_i| = 2$ and $d_H(B_i) \geq \zeta n$. Set
  \begin{align*}
    \mathcal{B}_{\zeta,b}(H) := \left\{ B \in \binom{V(H)}{b} : B \text{ is $\zeta$-separable}
    \right\}.
  \end{align*}
  If $H$ is obvious from context, we will denote this by $\mathcal{B}_{\zeta,b}$.
\end{definition}

The second main result proved in this section is that the property $(a,b,\epsilon,F)$-rich can be
replaced by $(a,\mathcal{B}_{\zeta,b},\epsilon,F)$-rich in
Proposition~\ref{prop:richimpliespacking}; the proof is in Section~\ref{sub:separable}.

\begin{prop} \label{prop:richseparableimpliespacking}
  Fix $0 < p,\alpha < 1$ and let $\zeta = \min\{\frac{p}{4}, \frac{\alpha}{4}\}$. Let $F$ be a
  $3$-graph such that $F$ is either linear or $k$-partite, let $v(F)|a$, and let $v(F)|b$ where in
  addition $b$ is even.  For any $\epsilon > 0$, there exists an $n_0$ and $\mu > 0$ such that the
  following holds. If $H$ is an $(a,\mathcal{B}_{\zeta,b},\epsilon,F)$-rich \npmua $3$-graph where
  $v(F)|n$, then $H$ has a perfect $F$-packing.
\end{prop}

Note that if $b$ is even, $H$ is a $3$-graph, and $\delta(H) \geq \alpha \binom{n}{2}$, then
Proposition~\ref{prop:richseparableimpliespacking} implies
Proposition~\ref{prop:richimpliespacking}.  The proofs of Propositions~\ref{prop:richimpliespacking}
and~\ref{prop:richseparableimpliespacking} use the absorbing technique of
R\"odl-Ruci\'nski-Szemer\'edi~\cite{pp-rodl09}.  The two key ingredients are the Absorbing Lemma
(Lemma~\ref{lem:absorbing}) and the Embedding Lemmas (Lemma~\ref{lem:manyatvertex} for linear and
Lemma~\ref{lem:supersat} for $k$-partite).  The remainder of this section contains the statements
and proofs of these lemmas plus the proofs of both propositions.

\subsection{Absorbing Sets} 
\label{sec:absorbing}

R\"odl-Ruci\'nski-Szemer\'edi~\cite[Fact 2.3]{pp-rodl09} have a slightly different definition of
edge-absorbing where $B$ has size $k+1$ and one vertex of $A$ is left out of the perfect matching,
but the main idea transfers to our setting in a straightforward way as follows.  If $H$ is a
$k$-graph, $A\subseteq V(H)$, and $\mathcal{A}\subseteq 2^{V(H)}$, then we say that $A$
\emph{partitions into sets from} $\mathcal{A}$ if there exists a partition $A = A_1 \dot\cup \cdots
\dot\cup A_t$ such that $A_i \in \mathcal{A}$ for all $i$.

\begin{lemma} (Absorbing Lemma) \label{lem:absorbing}
  Let $F$ be a $k$-graph, $\epsilon > 0$, and $a$ and $b$ be multiples of $v(F)$.  There
  exists an $n_0$ and $\omega > 0$ such that for all $n$-vertex $k$-graphs $H$ with $n \geq n_0$,
  the following holds.  If $H$ is $(\mathcal{A},\mathcal{B},\epsilon,F)$-rich for some $\mathcal{A}
  \subseteq \binom{V(H)}{a}$ and $\mathcal{B} \subseteq \binom{V(H)}{b}$, then there exists an
  $A \subseteq V(H)$ such that $A$ partitions into sets from $\mathcal{A}$ and $A$
  $F$-absorbs all sets $C$ satisfying the following conditions: $C \subseteq V(H)\setminus A$, $|C|
  \leq \omega n$, and $C$ partitions into sets from $\mathcal{B}$.
\end{lemma}

Using the idea of R\"odl-Ruci\'nski-Szemer\'edi~\cite{pp-rodl09}, Treglown and Zhao~\cite[Lemma
5.2]{pp-treglown09} proved the above lemma for $F$ a single edge, $a = 2k$, $b = k$, $\mathcal{A} =
\binom{V(H)}{a}$ and $\mathcal{B} = \binom{V(H)}{b}$.  For the sparse case (Theorem~\ref{thm:sparse})
we require a stronger version of Lemma~\ref{lem:absorbing} and so a proof of
Lemma~\ref{lem:absorbing} appears in Section~\ref{sec:Sparse} (as a corollary of
Lemma~\ref{lem:sparse-absorbing}).

\subsection{Embedding Lemmas and Almost Perfect Packings} 
\label{sub:embedding}

This section contains embedding lemmas for linear and $k$-partite $k$-graphs and a simple corollary
of these lemmas which produces a perfect $F$-packing covering almost all of the vertices.

\begin{definition}
  Let $F$ and $H$ be $k$-graphs with $V(F) = \{w_1,\dots,w_f\}$.  A \emph{labeled copy of $F$ in
  $H$} is an edge-preserving injection from $V(F)$ to $V(H)$.  A \emph{degenerate labeled copy of
  $F$ in $H$} is an edge-preserving map from $V(F)$ to $V(H)$ that is not an injection.  Let $1 \leq
  m \leq f$ and let $Z_1, \dots, Z_m \subseteq V(H)$.  Set $\inj[F \rightarrow H; w_1 \rightarrow
  Z_1, \dots, w_m \rightarrow Z_m]$ to be the number of edge-preserving injections $\psi : V(F)
  \rightarrow V(H)$ such that $\psi(w_i) \in Z_i$ for all $1 \leq i \leq m$.  In other words,
  $\inj[F \rightarrow H; w_1 \rightarrow Z_1, \dots, w_m \rightarrow Z_m]$ is the number of labeled
  copies of $F$ in $H$ where $w_i$ is mapped into $Z_i$ for all $1 \leq i \leq m$.  If $Z_i =
  \{z_i\}$, we abbreviate $w_i \rightarrow \{z_i\}$ as $w_i \rightarrow z_i$.
\end{definition}

\begin{lemma} \label{lem:manyatvertex}
  Let $0 < p,\alpha < 1$ and let $F$ be a linear $k$-graph where $0 \leq m \leq v(F)$ and $V(F) =
  \{s_1,\dots,s_m, \linebreak[1] t_{m+1},\dots,t_f\}$ such that there does not exist $E \in F$ with
  $|E \cap \{s_1,\dots,s_m\}| > 1$ and there do not exist $E_1, E_2 \in F$ with $|E_1 \cap
  \{s_1,\dots,s_m\}| = 1$, $|E_2 \cap \{s_1,\dots,s_m\}| = 1$, and $E_1 \cap E_2 \cap
  \{t_{m+1},\dots,t_f\} \neq \emptyset$.
  
  For every $\gamma > 0$, there exists an $n_0$  and $\mu > 0$ such that the following holds.  Let
  $H$ be an $n$-vertex $k$-graph with $n \geq n_0$ and let $y_1,\dots,y_m \in V(H), Z_{m+1}
  \subseteq V(H), \dots, Z_f \subseteq V(H)$.  Assume that for every $\{s_i,t_{j_2},\dots,t_{j_k}\}
  \in F$
  \begin{align} \label{eq:mindegatvertex}
    \left| \Big\{ \left( z_{j_2},\dots,z_{j_k} \right) \in Z_{j_2} \times \dots \times Z_{j_k} :
    \left\{ y_i, z_{j_2},\dots,z_{j_k} \right\} \in H \Big\} \right| \geq \alpha |Z_{j_2}| \cdots
    |Z_{j_k}|
  \end{align}
  and for every $\{t_{i_1},\dots,t_{i_k}\} \in F$ and every $Z'_{i_1} \subseteq Z_{i_1}, \dots,
  Z'_{i_k} \subseteq Z_{i_k}$,
  \begin{align} \label{eq:embedlineardense}
    e(Z'_{i_1},\dots,Z'_{i_k}) \geq p |Z_{i_1}| \cdots |Z_{i_k}| - \mu n^k.
  \end{align}
  Then 
  \begin{align*}
    \inj[F \rightarrow H;s_1 \rightarrow y_1, \dots, &s_m \rightarrow y_m, t_{m+1} \rightarrow
    Z_{m+1},\dots,t_f \rightarrow Z_f] \\
    &\geq 
    \alpha^{d_F(s_1)} \cdots \alpha^{d_F(s_m)} p^{|F|-\sum d_F(s_i)} |Z_{m+1}| \cdots |Z_f| -
    \gamma n^{f-m}.
  \end{align*}
\end{lemma}

\begin{proof} 
Kohayakawa, Nagle, R\"odl, and Schacht~\cite{hqsi-kohayakawa10} proved this when $Z_i = V(H)$ for
all $i$, without the distinguished vertices $s_1,\dots,s_m$, and under a stronger condition on $H$,
but it is straightforward to extend their proof to our setup as follows.  The lemma is proved by
induction on number of edges of $F$ which do not contain any vertex from among $s_1, \dots, s_m$.
Let $\mu = (1-p)\gamma$.

First, if every edge of $F$ contains some $s_i$ then $F$ is a vertex disjoint union of stars with centers
$s_1,\dots,s_m$
plus some isolated vertices.  Therefore, we can form a copy of $F$ of the type
we are trying to count by picking an edge of $H$ containing $y_i$ (of the right type) for each edge of
$F$.  More precisely, using \eqref{eq:mindegatvertex}, the fact that all edges of $F$ which use some
$s_1,\dots,s_m$ (so all edges of $F$) do not share any vertices from among $t_{m+1},\dots,t_f$, and
the fact that $F$ is linear,
the number of labeled copies of $F$ with $s_i \rightarrow y_i$ and
$t_j \rightarrow Z_j$ is at least
\begin{align*}
  \alpha^{|F|} |Z_{m+1}| \cdots |Z_f| = \alpha^{\sum d_F(s_i)} p^0 |Z_{m+1}| \cdots |Z_f|.
\end{align*}
The proof of the base case is complete.

Now assume $F$ has at least one edge $E$ which does not contain any $s_i$, with vertices labeled so
that $E = \{t_{m+1},\dots,t_{m+k}\}$.  Let $F_*$ be the hypergraph formed by deleting all vertices
of $E$ from $F$ and notice that $s_i \in V(F_*)$ for all $i$.  Let $F_-$ be the hypergraph formed by
removing the edge $E$ from $F$ but keeping the same vertex set.  Let $Q_*$ be an injective
edge-preserving map $Q_* : V(F_*) \rightarrow V(H)$ where $Q_*(s_i) = y_i$ for $1 \leq i \leq m$ and
$Q_*(t_j) \in Z_j$ for $m+1 \leq j \leq f$. For $m+1 \leq j \leq m+k$, define $S_j(Q_*) \subseteq
Z_j$ as follows.  For each $z \in Z_j$, add $z$ to $S_j(Q_*)$ if $z \notin Im(Q_*)$ and there exists
an edge-preserving injection $V(F_*) \cup \{t_j\} \rightarrow Im(Q_*) \cup \{z\}$ which when
restricted to $V(F_*)$ matches the map $Q_*$.  More informally, $S_j(Q_*)$ consists of all vertices
which can be used to extend $Q_*$ to embed a labeled copy of $F_* \cup \{t_j\}$.

By definition, every edge counted by $e(S_{m+1}(Q_*), \dots, S_{m+k}(Q_*))$ creates a labeled copy of
$F$.  Also, every ordered tuple from $S_{m+1}(Q_*) \times \dots \times S_{m+k}(Q_*)$ creates a labeled copy of
$F_-$.  More precisely,
\begin{align}
  \inj[F \rightarrow H; s_1\rightarrow y_1,\dots,s_m\rightarrow y_m, &t_{m+1}\rightarrow Z_{m+1},
  \dots, t_f\rightarrow Z_f] \nonumber \\ 
  &= \sum_{Q_*} e(S_{m+1}(Q_*), \dots, S_{m+k}(Q_*)) \nonumber \\
  \inj[F_- \rightarrow H; s_1\rightarrow y_1,\dots,s_m\rightarrow y_m, &t_{m+1}\rightarrow Z_{m+1},
  \dots, t_f\rightarrow Z_f] \nonumber \\
  &= \sum_{Q_*} |S_{m+1}(Q_*)| \cdots |S_{m+k}(Q_*)|. \label{eq:countFminus}
\end{align}
For each $j$, $S_j(Q_*) \subseteq Z_j$ so that \eqref{eq:embedlineardense} implies that
\begin{align}
  \inj[F \rightarrow H;s_1\rightarrow y_1,\dots,s_m\rightarrow y_m, &t_{m+1}\rightarrow Z_{m+1},
  \dots, t_f\rightarrow Z_f] \nonumber \\
  &\geq \sum_{Q_*} \left( p |S_{m+1}(Q_*)| \cdots |S_{m+k}(Q_*)| - \mu
  n^{k}\right) \nonumber \\
  &\geq p\sum_{Q_*} |S_{m+1}(Q_*)| \cdots |S_{m+k}(Q_*)| - \mu n^{f-m}, \label{eq:countFbyQsum}
\end{align}
where the last inequality is because there are at most $n^{f-m-k}$ maps $Q_*$, since $F_*$ has $f-k$
vertices and $s_i \in V(F_*)$ must map to $y_i$.  Combining \eqref{eq:countFminus} and
\eqref{eq:countFbyQsum} and then applying induction,
\begin{align*}
  \inj[F \rightarrow H; &s_1\rightarrow y_1,\dots,s_m\rightarrow y_m, t_{m+1}\rightarrow Z_{m+1},
  \dots, t_f\rightarrow Z_f] \\
  &\geq p \inj[F_- \rightarrow H; s_1\rightarrow y_1,\dots,s_m\rightarrow y_m, t_{m+1}\rightarrow
  Z_{m+1}, \dots, t_f\rightarrow Z_f] - \mu
  n^{f-m} \\
  &\geq p \left( \alpha^{\sum d(s_i)} p^{|F|-1-\sum d(s_i)}|Z_{m+1}|\cdots|Z_f|  - \gamma n^{f-m}  \right) - \mu
  n^{f-m}.
\end{align*}
Since $\mu = (1-p)\gamma$, the proof is complete.
\end{proof} 

\begin{cor} \label{cor:embedding}
  Let $0 < p < 1$ and let $F$ be a linear $k$-graph with $V(F) = \{t_1,\dots,t_f\}$.  For every
  $\gamma > 0$, there exists an $n_0$ and $\mu > 0$ such that the following holds.  Let $H$ be an
  \npmu $k$-graph and let $Z_1 \dots, Z_f \subseteq V(H)$.  Then
  \begin{align*}
    \inj[F\rightarrow H; t_1\rightarrow Z_1, \dots, t_f\rightarrow Z_f]
    \geq p^{|F|} |Z_1| \cdots |Z_f| - \gamma n^f.
  \end{align*}
\end{cor}

\begin{proof} 
Apply Lemma~\ref{lem:manyatvertex} with $m = 0$.  Since $H$ is $(p,\mu)$-dense,
\eqref{eq:embedlineardense} holds.  Also, \eqref{eq:mindegatvertex} is vacuous since $m = 0$.
\end{proof} 

\begin{lemma} \label{lem:supersat}
  Let $0 < p < 1$ and let $K_{t_1,\dots,t_k}$ be the complete $k$-partite, $k$-graph with part sizes
  $t_1, \dots, t_k$ and parts labeled by $T_1, \dots, T_k$.  For every $0 < \mu < \frac{p}{2}$,
  there exists $n_0$ and $0 < \xi < 1$  such that the following holds.  Let $H$ be an \npmu
  $k$-graph with $n \geq n_0$.  Then for any $X_1, \dots, X_k \subseteq V(H)$ with $|X_j| \geq
  (2\mu/p)^{1/k} n$ for all $j$, the number of labeled copies of $K_{t_1,\dots,t_k}$ in $H$ with
  $T_i \subseteq X_i$ for all $i$ is at least $\xi \prod |X_i|^{t_i}$.
\end{lemma}

\begin{proof} 
Let $H'$ be the $k$-graph on $\sum |X_i|$ vertices with vertex set $Y_1 \dot\cup \cdots \dot\cup
Y_t$ where the sets $Y_i$ are disjoint and $Y_i \cong X_i$ for all $i$.  Note that because the sets
$X_i$ might overlap, a vertex of $H$ might appear more than once in $H'$.  Make $y_1 \in Y_1, \dots,
y_k \in Y_k$ a hyperedge of $H'$ if $y_1,\dots,y_k$ are distinct vertices of $H$ and
$\{y_1,\dots,y_k\} \in H$.  Let $t = \sum t_i$.
Since $H$ is \dense,
\begin{align*}
  e(H') = e_H(X_1,\dots,X_k) \geq p \prod_i |X_i| - \mu n^k \geq p \left( \frac{2\mu}{p} \right) n^k
  -  \mu n^k = \mu n^k \geq \frac{\mu}{k^k} v(H')^k.
\end{align*}
Therefore, by supersaturation (see \cite[Theorems 2.1 and 2.2]{rrl-keevash11}), there exists an
$n'_0$ and $\xi' > 0$ such that if $v(H') \geq n'_0$ then $H'$ contains at least $\xi'
v(H')^{t}$ labeled copies of $K_{t_1,\dots,t_k}$.  Each of these labeled copies of
$K_{t_1,\dots,k_t}$ in $H'$ produces a possibly degenerate labeled copy of $K_{t_1,\dots,t_k}$ in
$H$ where $T_i \subseteq X_i$ for all $i$.  Pick $\xi = \frac{1}{2} \xi'$, $n_0 \geq n'_0
(p/2\mu)^{1/k}$, and $n_0 \geq \frac{1}{\xi} (p/2\mu)^{t/k}$.

Now assume that $n \geq n_0$.  This implies that $v(H') \geq |X_1| \geq (2\mu/p)^{1/k} n \geq
n'_0$ so that there are at least $\xi' v(H')^t$ labeled copies of $K_{t_1,\dots,t_k}$ in $H'$.
Therefore, the number of possibly degenerate labeled copies of $K_{t_1,\dots,t_k}$ in $H$ with $T_i
\subseteq X_i$ for all $i$ is at least
\begin{align} \label{eq:embedksat1}
  \xi' v(H')^t = \xi' \prod_i v(H')^{t_i} \geq \xi' \prod_i |X_i|^{t_i} = 2\xi \prod_i
  |X_i|^{t_i}.
\end{align}
Since there are at most $n^{t-1}$ degenerate labeled copies, by the choice of $n_0$ and since
$|X_i| \geq (2\mu/p)^{1/k} n$ for all $i$, the number of degenerate labeled copies is at most
\begin{align} \label{eq:embedksat2}
  n^{t-1} 
  = \frac{1}{n} \left( \frac{p}{2\mu} \right)^{t/k} \prod_i 
  \left[ \left( \frac{2\mu}{p} \right)^{1/k} n \right]^{t_i}
  \leq \frac{1}{n} \left( \frac{p}{2\mu} \right)^{t/k} \prod_i |X_i|^{t_i}
  \leq \xi \prod_i |X_i|^{t_i}.
\end{align}
Combining \eqref{eq:embedksat1} with \eqref{eq:embedksat2} shows that there are at least $\xi
\prod_i |X_i|^{t_i}$ labeled copies of $K_{t_1,\dots,t_k}$ with $T_i \subseteq X_i$ for all $i$,
completing the proof.
\end{proof} 

With these lemmas in hand, we can prove that if $H$ is \dense and $F$ is linear or
$k$-partite, then $H$ has an $F$-packing covering almost all the vertices of $H$.

\begin{lemma} (Almost Perfect Packing Lemma) \label{lem:greedypacking}
  Fix $0 < p < 1$ and a $k$-graph $F$ with $f$ vertices such that $F$ is either linear or
  $k$-partite. Let $v(F)|b$.  For any $0 < \omega < 1$, there exists $n_0$ and
  $\mu > 0$ such that the following holds.  Let $H$ be an \npmu $k$-graph with $n \geq n_0$ and
  $f|n$.  Then there exists $C \subseteq V(H)$ such that $|C| \leq \omega n$, $b| |C|$, and
  $H[\bar{C}]$ has a perfect $F$-packing.
\end{lemma}

\begin{proof} 
First, select $n_0$ large enough and $\mu$ small enough so that any vertex set $C$ of size
$\left\lceil \frac{\omega}{2} \right\rceil$ contains a copy of $F$.  To see that this is possible,
there are two cases to consider.

If $F$ is linear, let $\gamma = \frac{1}{2}p^{|F|}(\frac{\omega}{2})^f$ and select $n_0$ and $\mu
> 0$ according to Corollary~\ref{cor:embedding}.  Now if $C \subseteq V(H)$ with $|C| \geq \frac{\omega
}{2} n$, then Corollary~\ref{cor:embedding} implies there are at least $p^{|F|}|C|^f - \gamma n^f
\geq p^{|F|} \left( \frac{\omega}{2} \right)^f n^f - \gamma n^f = \gamma n^f > 0$ copies of $F$
inside $C$.

If $F$ is $k$-partite, then Lemma~\ref{lem:supersat} is used in a similar way as follows.  Let $\mu
= \frac{p}{2} \left( \frac{\omega}{2} \right)^k$ and select $n_0$ and $\xi$ according to
Lemma~\ref{lem:supersat}.  Now by the choice of $\mu$, if $|C| \geq \frac{\omega}{2}$ then $|C| \geq
(2\mu/p)^{1/k} n$ so that by Lemma~\ref{lem:supersat}, $C$ contains at least $\xi
(\frac{\omega}{2})^f n^f > 0$ copies of $F$.

Now let $F_1, \dots, F_t$ be a greedily constructed $F$-packing.  That is, $F_1, \dots, F_t$ are
disjoint copies of $F$ and $C := V(H) \setminus V(F_1) \setminus \dots \setminus V(F_t)$ has no copy
of $F$.  By the previous two paragraphs, $|C| \leq \frac{\omega}{2} n$.  Since $f|n$ and
$H[\bar{C}]$ has a perfect $F$-packing, $f| |C|$.  Thus we can let $y \equiv - \frac{|C|}{f} \pmod
{b}$ with $0\leq y < b$ and take $y$ of the copies of $F$ in the $F$-packing of $H[\bar{C}]$ and add
their vertices into $C$ so that $b| |C|$.
\end{proof} 

\subsection{Proof of Proposition~\ref{prop:richimpliespacking}} 
\label{sub:proof:richimpliespacking}

\begin{proof}[Proof of Proposition~\ref{prop:richimpliespacking}] 
First, select $\omega > 0$ according to Lemma~\ref{lem:absorbing} and $\mu_1 > 0$ accoding to
Lemma~\ref{lem:greedypacking}.  Also, make $n_0$ large enough so that both Lemma~\ref{lem:absorbing}
and~\ref{lem:greedypacking} can be applied.  Let $\mu = \mu_1 \omega^{k}$.  All the parameters have
now been chosen.

By Lemma~\ref{lem:absorbing}, there exists a set $A \subseteq V(H)$ such that $A$ $F$-absorbs $C$
for all $C \subseteq V(H) \setminus A$ with $|C| \leq \omega n$ and $b \mid |C|$.
If $|A| \geq (1-\omega)n$, then $A$ $F$-absorbs $V(H)\setminus A$ so that $H$ has a perfect
$F$-packing.  Thus $|A| \leq (1-\omega)n$.  Next, let $H' := H[\bar{A}]$ and notice that $H'$ is
$(p,\mu_1)$-dense since $v(H') \geq \omega n$ and
\begin{align*}
  \mu n^k \leq \frac{\mu}{\omega^k} v(H')^k = \mu_1 v(H')^k.
\end{align*}
Therefore, by Lemma~\ref{lem:greedypacking}, there exists a vertex set $C \subseteq V(H') = V(H)
\setminus A$ such that $|C| \leq \omega n$, $|C|$ is a multiple of $b$, and $H'[\bar{C}]$ has a perfect $F$-packing.
Now Lemma~\ref{lem:absorbing} implies that $A$ $F$-absorbs $C$.  The perfect $F$-packing of $A \cup
C$ and the perfect $F$-packing of $H'[\bar{C}]$ produces a perfect $F$-packing of $H$.
\end{proof} 

\subsection{Proof of Proposition~\ref{prop:richseparableimpliespacking}} 
\label{sub:separable}

This section contains the proof of Proposition~\ref{prop:richseparableimpliespacking}, but first we
need an extension of Lemma~\ref{lem:greedypacking} that produces a perfect $F$-packing covering
almost all the vertices where in addition the unsaturated vertices are $\zeta$-separable.  To do so, we
need a well-known probability lemma.

\begin{lemma} \label{lem:chernoff} \textbf{(Chernoff Bound)}
  Let $0 < p < 1$, let $X_1,\dots,X_n$ be mutually independent indicator random variables with
  $\mathbb{P}[X_i = 1] = p$ for all $i$, and let $X = \sum X_i$.
  Then for all $a > 0$,
  \begin{align*}
    \mathbb{P}[ \left| X - \mathbb{E}[X] \right| > a] \leq 2 e^{-a^2/2n}.
  \end{align*}
\end{lemma}

\begin{lemma} \label{lem:seppacking}
  Fix $p,\alpha \in (0,1)$, $\zeta = \min\{\frac{p}{4},\frac{\alpha}{4}\}$, and a $3$-graph $F$ such
  that either $F$ is linear or $F$ is $3$-partite. Let $v(F)|b$ where in addition $b$ is even.  For
  any $0 < \omega < 1$, there exists $n_0$ and $\mu > 0$ such that the following holds.  Let $H$ be
  an \npmua $3$-graph with $n \geq n_0$ and $v(F)|n$.  Then there exists a set $C \subseteq V(H)$
  such that $|C| \leq \omega n$, $C$ partitions into sets of $\mathcal{B}_{\zeta,b}$, and
  $H[\bar{C}]$ has a perfect $F$-packing.
\end{lemma}

\begin{proof} 
Use Lemma~\ref{lem:greedypacking} to select $n_0$ and $\mu_1 > 0$ to produce an $F$-packing $F_1,
\dots, F_t$ where $W := V(H) \setminus V(F_1)\setminus\dots\setminus V(F_t)$ is such that $|W| \leq
\frac{\omega \alpha}{4} n$.  Let $f = v(F)$ and let
\begin{align*}
  \phi &= \min \left\{ \frac{\omega}{8}, \frac{\alpha}{4}, \frac{p}{4}\right\}, \\
  \mu &= \min \left\{ \frac{p}{2} \left(\frac{\alpha \phi}{16f}\right)^2, \mu_1 \right\}.
\end{align*}
First, form a vertex set $C'$ by starting with $W$ and for each $1 \leq i \leq t$, add $V(F_i)$ to
$C'$ with probability $\phi$ independently.  After this, take $\frac{b}{f} - \frac{|C'|}{f} \pmod
{\frac{b}{f}}$ of the unselected copies of $F$ and add their vertices into $C'$ to form the vertex
set $C$.

By construction, $H[\bar{C}]$ has a perfect $F$-packing (the copies of $F$ which were not selected)
and $b| |C|$.  Since $b$ is even, $|C|$ is also even.  So to complete the proof,
we just need to show that with positive probability, $C$ is $\zeta$-separable and $|C| \leq \omega
n$.  (Note that if $C$ is $\zeta$-separable then it can be partitioned into sets from
$\mathcal{B}_{\zeta,b}$.)

Let $G$ be the graph where $V(G) = V(H)$ and for every $Z \in \binom{V(G)}{2}$, $Z$
is an edge of $G$ if $d_H(Z) \geq \zeta n$, i.e.\ the codegree of $Z$ in $H$ is at least $\zeta
n$.  We will now prove that with positive probability, the following two events occur:
\begin{itemize}
  \item $|C| \leq \frac{1}{2} \omega n$,
  \item $\delta(G[C]) \geq \frac{\alpha \phi}{8f} n$.
\end{itemize}
First, the expected number of vertices added to $W$ to form $C$ is $\phi ft \leq \frac{\omega}{8} n$
plus potentially a few copies of $F$ to make $b| |C|$.  By the second moment method, with
probability at least $\frac{1}{4}$, at most $\frac{\omega}{4}n$ vertices are added to $W$ so that
$|C| \leq \frac{1}{2} \omega n$.  Secondly, since $\delta(H) \geq \alpha n^2$ it is the case that
$\delta(G) \geq \frac{\alpha}{2} n$.  Indeed, if there was some vertex $x$ with $d_G(x) <
\frac{\alpha}{2} n$, then $d_H(x) \leq |N_G(x)| \cdot n + n \cdot \zeta n < (\frac{\alpha}{2} +
\zeta) n^2$, a contradiction to the fact that $\delta(H) \geq \alpha n^2$ and $\zeta \leq
\frac{\alpha}{2}$.  Since $|W| \leq \frac{\omega \alpha}{4}n$, we have that any vertex $x$ has at
least $\frac{\alpha}{2}n - \frac{\omega \alpha}{4}n > \frac{\alpha}{4}n$ neighbors in $G$ outside
$W$.  Since each $F_i$ has size $f$, the vertex $x$ therefore has a neighbor in $G$ inside at least
$\frac{\alpha}{4f}n$ of the copies of $F$.  Therefore, the expected size of $\{ y \in C: xy \in
E(G)\}$ is at least $\frac{\alpha \phi}{4f} n$ and by Chernoff's Inequality
(Lemma~\ref{lem:chernoff}),
\begin{align*}
  \mathbb{P}\left[ \left| \{y \in C : xy\in E(G)\} \right| < \frac{\alpha \phi}{8f}n\right] \leq e^{-c n}
\end{align*}
for some constant $c$.  Thus $n_0$ can be selected large enough so that with probability at most
$\frac{1}{4}$, there is some $x \in V(G)$ such that $\left| \{y \in C: xy \in E(G)\} \right| <
\frac{\alpha \phi}{8f}n$.  This implies that with probability at least $\frac{1}{2}$, $|C| \leq
\omega n$ and $\delta(G[C]) \geq \frac{\alpha \phi}{8f}n$.

To complete the proof, we will show that $\delta(G[C]) \geq \frac{\alpha \phi}{8f}n$ implies that
$G[C]$ has a perfect matching (which is equivalent to $C$ being $\zeta$-separable).  Divide $C$ into
two equal sized parts $C_1$ and $C_2$ (recall that $|C|$ is even since $b$ is even and $b| |C|$.
Such a partition exists since a random partition has this property with positive probability.
Assume towards a contradiction that Hall's Condition fails in $G[C_1,C_2]$, i.e.\ there exists a set
$T \subseteq C_1$ such that $|N_G(T) \cap C_2| < |T|$.  In a slight abuse of notation, let $\bar{T}
= C_1 \setminus T$.  Now $|T| \geq \frac{\alpha \phi}{16f} n$ since $\delta(G[C_1,C_2]) \geq
\frac{\alpha \phi}{16f}n$.  Similarly, $|\bar{T}| \geq \frac{\alpha \phi}{16f}n$ since if $z \in C_2
\setminus N_G(T)$ then $N_G(z) \cap C_1 \subseteq \bar{T}$.  This implies that
\begin{align*}
  |C_2 \setminus N_G(T)| = |C_2| - |N_G(T)| = |C_1| - |N_G(T)| > |C_1| - |T| = |\bar{T}| \geq \frac{\alpha
  \phi n}{16f}.
\end{align*}
Since there are no edges of $G$ between $T$ and $C_2 \setminus N_G(T)$,
\begin{align} \label{eq:halls-eq1}
  e_H(T, C_2 \setminus N_G(T), V(H)) \leq |T| \cdot |C_2 \setminus N_G(T)| \cdot \zeta n = \zeta |T|
  |C_2 \setminus N_G(T)| n.
\end{align}
On the other hand, since $H$ is \dense,
\begin{align*}
  e_H(T, C_2 \setminus N_G(T), V(H)) \geq p |T| |C_2 \setminus N_G(T)| n - \mu n^3.
\end{align*}
Since $|T|$ and $|C_2 \setminus N_G(T)|$ are both larger than $\frac{\alpha \phi}{16f} n$,
\begin{align*}
  e_H(T, C_2\setminus N_G(T), V(H)) \geq \left(p - \mu \left(\frac{16f}{\alpha \phi}\right)^2 \right) 
  |T| |C_2 \setminus N_G(T)| n
\end{align*}
Since $\mu \leq \frac{p}{2} (\frac{\alpha \phi}{16f})^2$, we have
\begin{align} \label{eq:halls-eq2}
  e_H(T, C_2\setminus N_G(T), V(H)) \geq \frac{p}{2} |T| |C_2 \setminus N_G(T)| n
\end{align}
Since $\zeta < \frac{p}{2}$, \eqref{eq:halls-eq2} contradicts \eqref{eq:halls-eq1}.  Therefore,
$G[C_1,C_2]$ satisfies Hall's condition so that $G[C]$ has a perfect matching, i.e.\ $C$ is
$\zeta$-separable.
\end{proof} 

\begin{proof}[Proof of Proposition~\ref{prop:richseparableimpliespacking}] 
The proof is similar to the proof of Proposition~\ref{prop:richimpliespacking} except
Lemma~\ref{lem:seppacking} is used instead of Lemma~\ref{lem:greedypacking}.
\end{proof} 

\section{Rich hypergraphs} 
\label{sec:existance}

This section contains the proofs of Theorems~\ref{thm:linearpacking}
and~\ref{thm:three-unif-packings}.  By the previous section, these proofs come down to showing that
$(p,\mu)$-dense and large minimum degree imply either $(a,b,\epsilon,F)$-rich or
$(a,\mathcal{B}_{\zeta,b},\epsilon,F)$-rich, where we get to select $a$, $b$, and $\epsilon$ but
$\zeta = \min\{\frac{p}{4}, \frac{\alpha}{4}\}$.  As a warm-up before
Theorem~\ref{thm:linearpacking} (see Section~\ref{sec:linear}), we start with the cherry.

\subsection{Packing Cherries} 
\label{sec:Cherry}

Let $K_{1,1,2}$ be the cherry.

\begin{lemma} \label{lem:cherry-absorb}
  Let $ 0 < p,\alpha < 1$ and let $\zeta = \min\{\frac{p}{4},\frac{\alpha}{4}\}$.  There exists an
  $n_0$, $\epsilon > 0$, and $\mu > 0$ such that  if $H$ is an \npmua $3$-graph with $n \geq n_0$,
  then $H$ is $(4,\mathcal{B}_{\zeta,4},\epsilon,K_{1,1,2})$-rich.
\end{lemma}

\begin{proof} 
Our main task is to come up with an $\epsilon > 0$ such that for large $n$ and all $B \in
\mathcal{B}_{\zeta,4}$, there are at least $\epsilon n^4$ vertex sets of size four which
$K_{1,1,2}$-absorb $B$; we will define $\epsilon$ and $\mu$ later.

Fix $B = \{b_1,b_2,b_3,b_4\} \in \mathcal{B}_{\zeta,4}$, labeled so that $d_H(b_1,b_2) \geq \zeta n$
and $d_H(b_3,b_4) \geq \zeta n$.  Let $X_1 = N(b_1,b_2) = \{ x : xb_1b_2 \in E(H)\} \subseteq
V(H)$ and $X_2 = N(b_3, b_4)$ and notice that $|X_1|, |X_2| \geq \zeta n$.  Arbitrarily divide
$X_1$ in half and call the two parts $Y_1$ and $Y_2$.  Let $\mu = \frac{p}{2} (\frac{\zeta}{2})^3$.
Since $|Y_1|, |Y_2|, |X_2| \geq \frac{\zeta }{2}n = (2\mu/p)^{1/3} n$, by Lemma~\ref{lem:supersat}
there exists a $\xi > 0$ and $n_0$ such that $H[Y_1, Y_2, X_2]$ contains at least $\xi (\frac{\zeta
n}{2})^4$ copies of $K_{1,1,2}$ with one degree two vertex in each of $Y_1$ and $Y_2$
and the degree one vertices in $X_2$.  The proof is now complete, since each of these cherries
absorbs $B$.  Indeed, let $\epsilon = \xi (\frac{\zeta}{2})^4$ and let $y_1 \in Y_1$, $y_2 \in Y_2$,
and $x_1, x_2 \in X_2$ be such that $y_1y_2x_1, y_1y_2x_2 \in E(H)$.  Then $A =
\{y_1,y_2,x_1,x_2\}$ $K_{1,1,2}$-absorbs $B$ because $b_1b_2y_1, b_1b_2y_2 \in E(H)$ (recall that
$Y_1,Y_2 \subseteq N(b_1,b_2)$) and similarly $b_3b_4x_1, b_3b_4x_2 \in E(H)$.  Since there are at
least $\epsilon n^4$ choices for $y_1, y_2, x_1, x_2$, the proof is complete.
\end{proof} 

\subsection{Packing Cycles} 
\label{sec:cycle}

Throughout this section, let $C_4$ denote the hypergraph $C_4(2+1)$.  This section completes the
proof of Theorem~\ref{thm:three-unif-packings}.

\begin{lemma} \label{lem:cycle-absorb}
  Let $ 0 < p,\alpha < 1$ and let $\zeta = \min\{\frac{p}{4},\frac{\alpha}{4}\}$.  There exists an
  $n_0$, $\epsilon > 0$, and $\mu > 0$ such that  if $H$ is a \npmua $3$-graph with $n \geq n_0$,
  then $H$ is $(18,\mathcal{B}_{\zeta,6},\epsilon,C_4)$-rich.
\end{lemma}

\begin{proof} 
Similar to the proof of Lemma~\ref{lem:cherry-absorb}, our task is to come up with an $\epsilon > 0$
such that for large $n$ and all $B \in \mathcal{B}_{\zeta,6}$, there are at least $\epsilon n^{18}$
vertex sets of size eighteen which $C_4$-absorb $B$; we will define $\epsilon$ and $\mu$ later.

Fix $B = \{b_1,b'_1,b_2,b'_2,b_3,b'_3\} \in \mathcal{B}_{\zeta,6}$ labeled so that $d_H(b_i,b'_i)
\geq \zeta n$ for all $i$.  For $1 \leq i \leq 3$, let $X_i = N(b_i,b'_i)$ and note that $|X_i| \geq
\zeta n$.  Now for each $1 \leq i \leq 3$, define
\begin{align*}
  R_i = \left\{ \{r_1, r_2\} \in \binom{V(H)}{2} : |N(r_1,r_2) \cap X_i| \geq \frac{1}{10} p\zeta
  n\right\}.
\end{align*}
In other words, $R_i$ is the set of pairs with neighborhood in $X_i$ at least one-tenth the
``expected'' size.  If $|R_1| \leq \frac{1}{10} p \zeta n^2$, then
\begin{align}\label{eq:c4edgebound}
  e(X_1,V(H),V(H)) \leq |R_1| n + \left( \binom{n}{2} - |R_1| \right)
  \frac{1}{10} p \zeta n \leq \frac{1}{5} p\zeta n^3.
\end{align}
On the other hand, since $H$ is $(p,\mu)$-dense,
\begin{align*}
  e(X_1,V(H),V(H)) \geq p |X_1|n^2 - \mu n^3 \geq \left( p\zeta - \mu \right) n^3.
\end{align*}
Let $\mu = \frac{p}{2} (\frac{1}{10}p\zeta)^3 < \frac{4}{5} p \zeta$ so that this contradicts
\eqref{eq:c4edgebound}.  Thus $|R_1| \geq \frac{1}{10} p\zeta n^2$ and similarly for $1 \leq i \leq
3$, $|R_i| \geq \frac{1}{10} p\zeta n^2$.

Now fix $r_1r'_1 \in R_1$, $r_2r'_2 \in R_2$, and $r_3r'_3 \in R_3$.  There are at least
$(\frac{1}{10} p\zeta)^3 n^6$ such choices.  For $1 \leq i \leq 3$ let $Y_i = N(r_i,r'_i) \cap X_i$
so $|Y_i| \geq \frac{1}{10} p\zeta n = (\frac{2\mu}{p})^{1/3} n$.  By Lemma~\ref{lem:supersat},
there exists a $\xi > 0$ such that there are at least $\xi (\frac{1}{10}p\zeta)^{12} n^{12}$ copies
of $K_{4,4,4}$ across $Y_1, Y_2, Y_3$.  Let $T_1, T_2, T_3$ be the three parts of $K_{4,4,4}$ with
$T_i \subseteq Y_i$ and let $T_i = \{y^{i}_1, y^{i}_2, y^{i}_3, y^{i}_4\}$.

Let $\epsilon = \xi (\frac{1}{10}p\zeta)^{15}$; we claim that there are at least $\epsilon n^{18}$
vertex sets of size $18$ which $C_4$-absorb $B$.  Indeed, $A := \{r_i,r'_i,y^i_j : 1 \leq i \leq
3, 1 \leq j \leq 4\}$ forms a $C_4$-absorbing $18$-set for $B$ as follows.  First, $A$ has a perfect
$C_4$-packing: one $C_4$ uses vertices $r_1, r'_1, y^1_1, y^1_2, y^2_3, y^3_3$, another uses
vertices $r_2, r'_2, y^2_1, y^2_2, y^3_4, y^1_3$, and the last uses $r_3, r'_3, y^3_1, y^3_2, y^1_4,
y^2_4$.  Secondly, $A \cup B$ has a perfect $C_4$-packing: one $C_4$ using $b_1, b'_1, r_1, r'_1,
y^1_1, y^1_2$, one using $b_2, b'_2, r_2, r'_2, y^2_1, y^2_2$, one using $b_3, b'_3, r_3, r'_3,
y^3_1, y^3_2$, and one using $y^1_3, y^1_4, y^2_3, y^2_4, y^3_3, y^3_4$.  Since there are
$(\frac{1}{10} p\zeta)^3 n^6$ choices for $r_1, r'_1, r_2, r'_2, r_3, r'_3$ and then
$\xi(\frac{1}{10}p\zeta)^{12} n^{12}$ choices for $y^i_j$, there are a total of at least $\epsilon
n^{18}$ choices for $A$.
\end{proof} 

\begin{proof}[Proof of Theorem~\ref{thm:three-unif-packings}] 
  Apply Lemmas~\ref{lem:cherry-absorb} and~\ref{lem:cycle-absorb} and then
  Proposition~\ref{prop:richseparableimpliespacking}.
\end{proof} 

\subsection{Packing Linear Hypergraphs} 
\label{sec:linear}

In this section, we prove Theorem~\ref{thm:linearpacking}.

\begin{lemma} \label{lem:linear-absorb}
  Let $ 0 < p,\alpha < 1$ and let $F$ be a linear $k$-graph on $f$ vertices.  There exists an $n_0$,
  $\epsilon > 0$, and $\mu > 0$ such that if $H$ is a \npmua $k$-graph with $n \geq n_0$, then $H$
  is $(f^2-f,f,\epsilon,F)$-rich.
\end{lemma}

\begin{proof} 
Let $a = f(f-1)$ and $b = f$.  Similar to the proofs in the previous two sections, our task is to
come up with an $\epsilon > 0$ such that for large $n$ and all $B \in \binom{V(H)}{b}$, there are
at least $\epsilon n^a$ vertex sets of size $a$ which $F$-absorb $B$; we will define $\epsilon$ and
$\mu$ later.   Let $V(F) = \{w_0,\dots,w_{f-1}\}$ and form the following $k$-graph $F'$.  Let
\begin{align*}
  V(F') = \{ x_{i,j} : 0 \leq i, j \leq f-1 \}.
\end{align*}
(We think of the vertices of $F'$ as arranged in a grid with $i$ as the row and $j$ as the column.)
Form the edges of $F'$ as follows: for each fixed $1 \leq i \leq f-1$, let $\{x_{i,0},\dots,x_{i,f-1}\}$
induce a copy of $F$ where $x_{i,j}$ is mapped to $w_{i+j\pmod f}$.  More precisely, if
$\{w_{\ell_1},\dots,w_{\ell_k}\} \in F$, then $\{x_{i,\ell_1-i\pmod f}, \dots, x_{i,\ell_k-i\pmod
f}\} \in F'$.  Similarly, for each fixed $0 \leq j \leq f-1$, let $\{x_{0,j},\dots,x_{f-1,j}\}$
induce a copy of $F$ where $x_{i,j}$ is mapped to $w_{i+j\pmod f}$.  Note that we therefore have
a copy of $F$ in each column and a copy of $F$ in each row besides the zeroth row.

Now fix $B = \{b_0, \dots, b_{f-1}\} \subseteq V(H)$; we want to show that $B$ is $F$-absorbed by many
$a$-sets.  Note that any labeled copy of $F'$ in $H$ which maps $x_{0,0} \rightarrow b_0, \dots,
x_{0,f-1} \rightarrow b_{f-1}$ produces an $F$-absorbing set for $B$ as follows.  Let $Q : V(F')
\rightarrow V(H)$ be an edge-preserving injection where $Q(b_j) = x_{0,j}$ (so $Q$ is a labeled copy
of $F'$ in $H$ where the set $B$ is the zeroth row of $F'$).  Let $A = \{Q(x_{i,j}) : 1 \leq i \leq
f-1, 0 \leq j \leq f-1\}$ consist of all vertices in rows $1$ through $f-1$.  Then $A$ has a perfect
$F$-packing consisting of the copies of $F$ on the rows, and $A \cup B$ has a perfect $F$-packing
consisting of the copies of $F$ on the columns.  Therefore, $A$ $F$-absorbs $B$.

To complete the proof, we therefore just need to use Lemma~\ref{lem:manyatvertex} to show there are
many copies of $F'$ with $B$ as the zeroth row.
Apply Lemma~\ref{lem:manyatvertex} to $F'$ where $m = f$, $s_1 =
x_{0,0},\dots,s_f = x_{0,f-1}$ and $Z_{m+1} = \dots = Z_{f^2} = V(H)$.  Since $\delta(H) \geq \alpha \binom{n}{k-1}$,
\eqref{eq:mindegatvertex} holds (with $\alpha$ replaced by $\frac{\alpha}{k^k}$) and since $H$ is
$(p,\mu)$-dense \eqref{eq:embedlineardense} holds. Let $\gamma =
\frac{1}{2}(\frac{\alpha}{k^k})^{\sum d(x_{0,j})} p^{|F|-\sum d(x_{0,j})}$ and ensure that $n_0$ is large enough
and $\mu$ is small enough apply Lemma~\ref{lem:manyatvertex} to show that
\begin{align*}
  \inj[F' \rightarrow H; x_{0,0} \rightarrow b_0, \dots, x_{0,f-1} \rightarrow b_{f-1}] \geq \gamma
  n^{f^2-f} = \gamma n^a.
\end{align*}
Each labeled copy of $F'$ produces a labeled $F$-absorbing set for $B$, so there are at least
$\frac{\gamma}{a!} n^a$ $F$-absorbing sets for $B$.  The proof is complete by letting $\epsilon =
\frac{\gamma}{a!}$.
\end{proof} 

\begin{proof}[Proof of Theorem~\ref{thm:linearpacking}] 
Apply Lemma~\ref{lem:linear-absorb} and then Proposition~\ref{prop:richimpliespacking}.
\end{proof} 

\section{Avoiding perfect $F$-packings} 
\label{sec:constr}

In this section we prove Theorem~\ref{thm:nopacking} using the following construction.

\begin{constr}
  For $n \in \mathbb{N}$, define a probability distribution $H(n)$ on $3$-uniform, $n$-vertex
  hypergraphs as follows.  Let $G = G^{(2)}(n,\frac{1}{2})$ be the random graph on $n$ vertices.
  Let $X$ and $Y$ be a partition of $V(G)$ where
  \begin{itemize}
    \setlength{\itemsep}{1pt}
    \setlength{\parskip}{0pt}
    \setlength{\parsep}{0pt}
    \item if $n \equiv 0 \pmod 4$, then $|X| = \frac{n}{2} - 1$ and $|Y| = \frac{n}{2} + 1$,
    \item if $n \equiv 1 \pmod 4$, then $|X| = \frac{n}{2} - \frac{1}{2}$ and $|Y| = \frac{n}{2} +
      \frac{1}{2}$,
    \item if $n \equiv 2 \pmod 4$, then $|X| = |Y| = \frac{n}{2}$,
    \item if $n \equiv 3 \pmod 4$, then $|X| = \frac{n}{2} - \frac{1}{2}$ and $|Y| = \frac{n}{2} +
      \frac{1}{2}$.
  \end{itemize}
  Let the vertex set of $H(n)$ be $V(G)$ and make a set $E \in \binom{V(G)}{3}$ into a hyperedge of
  $H(n)$ as follows.  If $|E \cap X|$ is even, then make $E$ into a hyperedge of $H(n)$ if $G[E]$ is
  a clique.  If $|E \cap X|$ is odd, then make $E$ into a hyperedge of $H(n)$ if $E$ is an
  independent set in $G$.
\end{constr}

\begin{lemma} \label{lem:constrdense}
  For every $\epsilon > 0$, with probability going to $1$ as $n$ goes to infinity,
  \begin{align*}
    \left| |E(H(n))| - \frac{1}{8} \binom{n}{3} \right|  \leq \epsilon n^3.
  \end{align*}
\end{lemma}

\begin{proof} 
For each $E \in \binom{V(H(n))}{3}$, $E$ is a clique or independent set in $G(n,\frac{1}{2})$ with
probability $\frac{1}{8}$.  Thus the expected number of edges in $H(n)$ is $\frac{1}{8}
\binom{n}{3}$ so the second moment method shows that with probability going to one as $n$ goes to
infinity, $|E(H(n)) - \frac{1}{8} \binom{n}{3}| \leq \epsilon n^3$.  See \cite{prob-alonspencer} or
the proof of Lemma 15 in \cite{hqsi-lenz-poset12} for details about the second moment method.
\end{proof} 

\begin{lemma} \label{lem:constrexpand}
  For every $\epsilon > 0$, with probability going to $1$ as $n$ goes to infinity the following
  holds.  Let $X_1, X_2, X_3 \subseteq V(H(n))$.  Then
  \begin{align*}
    \left| e(X_1,X_2,X_3) - \frac{1}{8} |X_1| |X_2| |X_3| \right| < \epsilon n^3.
  \end{align*}
\end{lemma}

\begin{proof} 
Let $S_1, \dots, S_n$ be Steiner triple systems that partition $\binom{V(H(n))}{3}$.  That is, view
$V(H(n)) \cong \mathbb{Z}_{n-1}$ and for $1 \leq i \leq n$ let $S_{i+1}$ consist of the triples
$\{a,b,c\}$ such that $a + b + c = i \pmod n$.  Each triple in $\binom{V(H(n))}{3}$ appears in
exactly one $S_i$ and two triples from the same $S_i$ share at most one vertex.

Let $1 \leq i \leq n$ and let $X_1, X_2, X_3 \subseteq V(H(n))$.  Let $e_H(X_1,X_2,X_3;S_i)$ be the
number of ordered tuples $(x_1,x_2,x_3) \in X_1 \times X_2 \times X_3$ such that $\{x_1,x_2,x_3\}
\in E(H(n)) \cap S_i$.  Let $e_{K_n}(X_1,X_2,X_3;S_i)$ be the number of ordered tuples
$(x_1,x_2,x_3) \in X_1 \times X_2 \times X_3$ such that $\{x_1,x_2,x_3\} \in S_i$.

The expected value of $e_H(X_1,X_2,X_e;S_i)$ is clearly $\frac{1}{8} e_{K_n}(X_1,X_2,X_3;S_i)$.
If $E_1, E_2 \in S_i$ then since $E_1$ and $E_2$ share at most one vertex the events $E_1
\in E(H(n))$ and $E_2 \in E(H(n))$ are independent.  By Chernoff's Bound (Lemma~\ref{lem:chernoff}),
\begin{align*}
  \mathbb{P}\left[ \left| e_H(X_1,X_2,X_3;S_i) - \frac{1}{8}e_{K_n}(X_1,X_2,X_3;S_i) \right| >
  \epsilon |S_i| \right] < e^{-c n^2}.
\end{align*}
for some constant $c$ since $|S_i| = \frac{1}{n} \binom{n}{3}$ and the number of events is
$e_{K_n}(X_1,X_2,X_3;S_i) < |S_i|$.  By the union bound,
\begin{align*}
  \mathbb{P}\left[ \exists i, \exists X_1,X_2,X_3, \left| e_H(X_1,X_2,X_3;S_i) -
  \frac{1}{8}e_{K_n}(X_1,X_2,X_3;S_i) \right| > \epsilon |S_i| \right] < e^{-\frac{c}{2} n^2}.
\end{align*}
Therefore, with high probability, for all $i$ and all $X_1,X_2,X_3$,
\begin{align}\label{eq:constrexpand}
  \left|  e_H(X_1,X_2,X_3;S_i) - \frac{1}{8} e_{K_n}(X_1,X_2,X_3;S_i) \right| < \epsilon |S_i|.
\end{align}
Summing \eqref{eq:constrexpand} over $i$ completes the proof.
\end{proof} 

\begin{lemma} \label{lem:constrnopacking}
  Let $F$ be a $3$-graph with an even number of vertices such that there exists a partition of the
  vertices of $F$ into pairs such that every pair has a common pair in their links.  Then $H(n)$
  does not have a perfect $F$-packing for any $n$.
\end{lemma}

\begin{proof} 
If $n \nmid v(F)$, then obviously $H(n)$ does not have a perfect $F$-packing.  Therefore assume that
$n|v(F)$ so that $n$ is even.  Let $X$ and $Y$ be the partition of $V(H(n))$ in the definition of
$H(n)$.  Since $n$ is even, by definition both $|X|$ and $|Y|$ are odd.  Let $\{w_1,z_1\},
\{w_2,z_2\}, \dots, \{w_{v(F)/2}, z_{v(F)/2}\}$ be the partition of $V(F)$ into pairs so that
$w_i$ and $z_i$ have a common pair in their link for all $i$.  By construction, if $x \in X$ and $y
\in Y$ then there is no pair of vertices $u,v\in V(H(n))$ such that $xuv, yuv \in E(H(n))$ since the
parities of $\{x,u,v\} \cap X$ and $\{y,u,v\} \cap X$ are different.  This implies that for each
$i$, $w_i$ and $z_i$ must either both appear in $X$ or both appear in $Y$ so that any copy of $F$ in
$H(n)$ uses an even number of vertices in $X$ and an even number of vertices in $Y$.  Since $|X|$ is
odd, $H(n)$ does not have a perfect $F$-packing.
\end{proof} 

\begin{proof}[Proof of Theorem~\ref{thm:nopacking}] 
By Lemmas~\ref{lem:constrdense},~\ref{lem:constrexpand}, and~\ref{lem:constrnopacking}, with high
probability $H(n)$ has the required properties.
\end{proof} 

\section{Perfect Matchings in Sparse Hypergraphs} 
\label{sec:Sparse}

In this section, we prove Theorem~\ref{thm:sparse}.  We follow the same outline as
Section~\ref{sec:existance}.

\begin{lemma} \label{lem:sparse-absorbing}
  Let $k \geq 2$, $c > 0$, and $a,b$ be multiples of $k$.  There exists an $n_0$ depending only on
  $k$, $a$, $b$, and $c$ such that the following holds for all $n \geq n_0$.  Let $H$ be an
  $n$-vertex $k$-graph, let $\mathcal{A} \subseteq \binom{V(H)}{a}$, and let $\mathcal{B} \subseteq
  \binom{V(H)}{b}$.  Suppose that $\ell \geq c n^{a-1/2} \log n$ is an integer such that for every
  $B \in \mathcal{B}$ there are at least $\ell$ sets in $\mathcal{A}$ which edge-absorb $B$.  Then
  there exists set $A \subseteq V(H)$ such that $A$ partitions into sets from $\mathcal{A}$ and $A$
  edge-absorbs any set $C$ satisfying the following conditions: $C \subseteq V(H) \setminus A$, $|C|
  \leq \frac{1}{64} \ell^2 n^{-2a+1}$, and $C$ partitions into sets from $\mathcal{B}$.
\end{lemma}

\begin{proof} 
The proof is similar to Treglown-Zhao~\cite[Lemma 5.2]{pp-treglown09} which in turn is similar to
R\"odl-Ruci\'nski-Szemer\'edi~\cite[Fact 2.3]{pp-rodl09}. Let $q = \frac{1}{8} \ell n^{-2a+1}$ and
let $\mathfrak{A} \subseteq \mathcal{A}$ be the family obtained by selecting each element of
$\mathcal{A}$ with probability $q$ independently.  The expected number of intersecting pairs of
elements from $\mathfrak{A}$ is at most $q^2 \binom{n}{a} a \binom{n}{a-1} \leq \frac{1}{16} q\ell$.
By Markov's inequality, with probability at least $\frac{1}{2}$ there are at most $\frac{1}{8}
q\ell$ intersecting pairs of elements from $\mathfrak{A}$. 

Now fix $B \in \mathcal{B}$ and let $\Gamma_B \subseteq \{ A \in \mathcal{A}: A \text{
edge-absorbs } B\}$ be such that $|\Gamma_B| = \ell$. For each $A \in \Gamma_B$, let $X_A$ be the
event that $A \in \mathfrak{A}$.  By Chernoff's Bound (Lemma~\ref{lem:chernoff}),
\begin{align*}
  \mathbb{P}\left[ \Big| |\Gamma_B \cap \mathfrak{A}| - q\ell \Big| > \frac{1}{2}q\ell \right]
  \leq 2e^{-q\ell/6}.
\end{align*}
Using that $\ell \geq c n^{a-1/2} \log n$, we have that $q\ell = \frac{1}{8} \ell^2 n^{-2a+1} \geq
\frac{c^2}{8} \log^2 n$.  By the union bound,
\begin{align*}
  \mathbb{P}\left[ \exists B, \left| \Gamma_B \cap \mathfrak{A} \right| < \frac{1}{2} q\ell \right] \leq
  \binom{n}{b} 2e^{-q\ell/6}
  \leq 2e^{b \log n - c^2 \log^2 n/48} < \frac{1}{2}
\end{align*}
for large $n$.  Thus with probability at least $\frac{1}{2}$, $\mathfrak{A}$ is such that for all $B
\in \mathcal{B}$, there exist at least $\frac{1}{4}q\ell$ $a$-sets in $\mathfrak{A}$ which
edge-absorb $B$.  Also, with probability at least $\frac{1}{2}$ there are at most $\frac{1}{8}
q\ell$ intersecting pairs of elements from $\mathfrak{A}$.

Let $\mathfrak{A}'$ be the subfamily of $\mathfrak{A}$ consisting only of those $a$-sets $A$ where
$A$ is not in any intersecting pair and also there is at least one $B \subseteq V(H)$ (of any size)
such that $A$ edge-absorbs $B$.  Thus by the union bound, with positive probability $\mathfrak{A}'$
is such that for all $B \in \mathcal{B}$, there exist at least $\frac{1}{8}q\ell$ $a$-sets in
$\mathfrak{A}'$ which edge-absorb $B$.  Let $\mathfrak{A}'$ be such a family of $a$-sets and let $A'
= \cup \mathfrak{A}'$.  First, $H[A']$ has a perfect matching.  Indeed, each $A \in \mathfrak{A}'$
edge-absorbs some set so $H[A]$ has a perfect matching, and the sets in $\mathfrak{A}'$ are disjoint
so that these perfect matchings combine to form a perfect matching of $H[A']$.  Second, $A'$
partitions into sets from $\mathcal{A}$ since the sets in $\mathfrak{A}'\subseteq \mathcal{A}$ are
disjoint. Now let $C \subseteq V(H) \setminus A'$ with $|C| \leq \frac{1}{64} \ell^2 n^{-2a+1} =
\frac{1}{8}q\ell$ and $C = B_1 \dot\cup \cdots \dot\cup B_t$ with $B_i \in \mathcal{B}$. Using the
bound on the size of $C$, we have that $t < \frac{1}{8} q\ell$.  Since each $B_i$ is edge-absorbed
by at least $\frac{1}{8}q\ell$ sets in $\mathfrak{A}'$, each $B_i$ can be edge-absorbed by a
different $a$-set in $\mathfrak{A}'$.  Therefore, $H[A' \cup B']$ has a perfect matching so the
proof is complete.
\end{proof} 

\begin{proof}[Proof of Lemma~\ref{lem:absorbing}] 
Let $H'$ be the $v(F)$-uniform hypergraph on the same vertex set as $H$, where $X \in
\binom{V(H)}{v(F)}$ is a hyperedge of $H'$ if $H[X]$ is a copy of $F$.  Let $\ell = \left\lceil
\epsilon n^a \right\rceil$ and notice since $H$ is $(\mathcal{A},\mathcal{B},\epsilon,F)$-rich, for
every $B \in \mathcal{B}$ there are at least $\ell$ sets in $\mathcal{A}$ which edge-absorb $B$ in
$H'$.  Also, since $a,b$ are multiples of $v(F)$ they are multiples of the uniformity of $H'$.
Lastly, for large $n$ we have that $\ell \geq n^{a-1/2}\log n$.  Therefore, applying
Lemma~\ref{lem:sparse-absorbing} (with $c = 1$) to $H'$ shows that there exists a set $A \subseteq
V(H') = V(H)$ such that $A$ partitions into sets from $\mathcal{A}$ and for any $C \subseteq V(H')
\setminus A = V(H)\setminus A$ with $|C| \leq \frac{\epsilon^2}{64} n$ and $C$ partitions into sets
from $\mathcal{B}$, $A$ edge-absorbs $C$ in $H'$.  Because each edge of $H'$ is a copy of $F$, this
implies that $A$ $F$-absorbs $C$ in $H$ so the proof is complete by setting $\omega =
\frac{\epsilon^2}{64}$.
\end{proof} 

Next, similar to the proofs in Sections~\ref{sec:Cherry},~\ref{sec:cycle}, and~\ref{sec:linear}, we
show that a bound on  $\lambda_2(H)$ implies that each $3$-set is edge-absorbed by many $6$-sets.
To do so, we need the hypergraph expander mixing lemma, first proved by Friedman and
Wigderson~\cite{ee-friedman95,ee-friedman95-2} (using a slightly different definition of
$\lambda_2(H)$) and then extended to our definition of $\lambda_2(H)$ in~\cite{hqsi-lenz-quasi12}.

\begin{prop} \label{prop:expandermixing}
  (Hypergraph Expander Mixing Lemma~\cite[Theorem 4]{hqsi-lenz-quasi12}).  Let $H$ be an $n$-vertex
  $k$-graph and let $S_1, \dots, S_k \subseteq V(H)$.  Then
  \begin{align*}
    \left| e(S_1, \dots, S_k) - \frac{k! |E(H)|}{n^k} \prod_{i=1}^k |S_i|  \right|
    \leq \lambda_2(H) \sqrt{|S_1| \cdots |S_k|}.
  \end{align*}
\end{prop}

\begin{lemma} \label{lem:sparse-edge-absorb}
  Let $\alpha > 0$.  Let $H$ be a $3$-graph and let $p = 6|E(H)|/n^3$.  Assume $\delta_2(H) \geq
  \alpha p n$ and $\lambda_2(H) \leq \frac{1}{2}\alpha^2 p^{5/2} n^{3/2}$.  Then for every $B
  \subseteq V(H)$ with $|B| = 3$, there are at least $\frac{1}{16} \alpha^4 p^5 n^6$ sets $A
  \subseteq V(H)$ with $|A| = 6$ such that $A$ edge-absorbs $B$.
\end{lemma}

\begin{proof} 
Let $B = \{b_1,b_2,b_3\} \subseteq V(H)$.  First, there are at least $\frac{1}{8} \alpha pn^3$ edges
disjoint from $B$; let $\{x_1,x_2,x_3\}$ be such an edge.  For $1 \leq i \leq 3$, let $Y_i \subseteq
N(b_i, x_i) = \{y : yx_ib_i \in H\}$ with $|Y_i| = \alpha pn$.  Such a $Y_i$ exists since the
minimum codegree is at least $\alpha pn$.  By the expander mixing lemma
(Proposition~\ref{prop:expandermixing}),
\begin{align*}
  e(Y_1,Y_2,Y_3) \geq p|Y_1||Y_2||Y_3| - \lambda_2(H) \sqrt{|Y_1||Y_2||Y_3|}
  \geq \alpha^3 p^4 n^3 - \lambda_2(H) \alpha^{3/2} p^{3/2} n^{3/2}.
\end{align*}
Since $\lambda_2(H) \leq \frac{1}{2} \alpha^2 p^{5/2} n^{3/2}$,
\begin{align*}
  e(Y_1,Y_2,Y_3) \geq \alpha^3 p^4 n^3 - \frac{1}{2}\alpha^{7/2} p^4 n^3 \geq \frac{1}{2} \alpha^3 p^4 n^3.
\end{align*}
Let $\{y_1,y_2,y_3\}$ be an edge with $y_1 \in Y_1$, $y_2 \in Y_2$, and $y_3 \in Y_3$.  Then
$\{x_1,x_2,x_3,y_1,y_2,y_3\}$ is a six-set that edge-absorbs $B$ and there are at least
$\frac{1}{8}\alpha pn^3 \cdot \frac{1}{2} \alpha^3 p^4 n^3 = \frac{1}{16} \alpha^4 p^5 n^6$ such
sets.
\end{proof} 

\begin{proof}[Proof of Theorem~\ref{thm:sparse}] 
We are given $\alpha > 0$ such that $\delta_2(H) \geq \alpha pn$.  Let $\gamma = 2^{-22}
\alpha^{12}$.

First, we can assume that $p \geq \gamma n^{-1/10} \log^{1/5} n$.  Indeed, by averaging there exists
vertices $s_1, s_2$ such that the codegree of $s_1$ and $s_2$ is at most $2pn$.  Then taking $S_1 =
\{s_1\}$, $S_2 = \{s_2\}$, and $S_3$ as the non-coneighbors of $s_1$ and $s_2$,
Proposition~\ref{prop:expandermixing} shows that
\begin{align*}
  \lambda_2(H) \geq p \sqrt{|S_3|} \geq p \sqrt{(1-2p) n}.
\end{align*}
But by assumption, $\lambda_2(H) \leq \gamma p^{16} n^{3/2}$.  Therefore,
\begin{align*}
  p \sqrt{(1-2p)n} \leq \gamma p^{16} n^{3/2}
\end{align*}
which implies that $p \geq \gamma n^{-1/10} \log^{1/5} n$ (by a large margin).

By Lemma~\ref{lem:sparse-edge-absorb}, for every $B \subseteq V(H)$ with $|B| = 3$ there are at
least $\frac{1}{16} \alpha^4 p^5 n^6$ $6$-sets $A \subseteq V(H)$ which edge-absorb $B$.  Let $\ell
= \frac{1}{16} \alpha^4 p^5 n^6$.  If $n$ is sufficiently large, then since $p \geq \gamma n^{-1/10}
\log^{1/5} n$, we have that $\ell = \frac{1}{16} \alpha^4 p^5 n^6 \geq \frac{1}{16} \alpha^4
\gamma^5
n^{5.5} \log n$.  Let $c = \frac{1}{16} \alpha^4 \gamma^5$ so that $\ell \geq c n^{5.5} \log n$.  Now by
Lemma~\ref{lem:sparse-absorbing}, if $n$ is sufficiently large there exists $A \subseteq V(H)$
such that $A$ edge-absorbs all sets of size a multiple of three and at most
\begin{align} \label{eq:sparse-size-absorb}
  \frac{1}{64} \ell^2 n^{-11} = \frac{1}{2^{14}} \alpha^8 p^{10} n.
\end{align}
We now show how to construct a perfect matching in $H$.  First, greedily construct a matching in
$H[V(H) \setminus A]$.  Say the greedy procedure halts with $B \subseteq V(H) \setminus A$ as the
unmatched vertices.  Since $3|v(H)$ and $3| |A|$ (since $A$ is an edge-absorbing set), $3| |B|$.  By
Proposition~\ref{prop:expandermixing} (recall that $\gamma = 2^{-22} \alpha^{12}$),
\begin{align} \label{eq:sparse-eBBB}
  e(B,B,B) = p|B|^3 \pm \lambda_2(H) |B|^{3/2} \geq p|B|^3 - \frac{1}{2^{22}}\alpha^{12} p^{16} n^{3/2} |B|^{3/2}.
\end{align}
If $|B| \geq 2^{-14} \alpha^8 p^{10} n$, then
\begin{align*}
  p|B|^3 \geq p |B|^{3/2} \left( \frac{1}{2^{14}} \alpha^8 p^{10} n \right)^{3/2}
  = \frac{1}{2^{21}} \alpha^{12} p^{16} n^{3/2} |B|^{3/2}.
\end{align*}
Combining this with \eqref{eq:sparse-eBBB} shows that $e(B,B,B) > 0$.  This contradicts
that the greedy procedure halted with $B$ as the unmatched vertices.  Thus $|B| \leq
2^{-14} \alpha^8 p^{10} n$ and then \eqref{eq:sparse-size-absorb} shows that $A$
edge-absorbs $B$, producing a perfect matching of $H$.
\end{proof} 

\medskip

\textit{Acknowledgements:} We would like to thank Alan Frieze for helpful discussions at the early
stages of this project.  Thanks also to Daniela K\"uhn and Sebastian Cioab{\u{a}} for useful
comments.

\bibliographystyle{abbrv}
\bibliography{refs.bib}

\end{document}